\documentclass[12pt]{article}
\usepackage{amsmath,amssymb,graphicx,amsthm}
\usepackage{subfigure}

\newtheorem{theorem}{Theorem}[section]
\newtheorem{lemma}[theorem]{Lemma}
\newtheorem{obs}[theorem]{Observation}
\newtheorem{corollary}[theorem]{Corollary}
\newtheorem{problem}[theorem]{Problem}
\newtheorem{conjecture}[theorem]{Conjecture}
\newtheorem{claim}[theorem]{Claim}
\newtheorem{definition}[theorem]{Definition}

\newcommand\cref[1]{Corollary~\ref{cor:#1}}

\newcommand\chol[0]{\chi_{P}}
\newcommand\cholb[0]{\chi_{P,b}}
\newcommand\cholf[0]{\chi_{P,f}}
\newcommand{\olch}{\chi_P}
\newcommand{\NN}{\mathbb{N}}

\DeclareMathOperator{\ch}{ch}

\usepackage{color}

\begin{document}

\title{Choosability and paintability of the lexicographic product of graphs}
\author{Bal\'azs Keszegh\thanks{Alfr\'ed R\'enyi Institute of
Mathematics, P.O.B. 127, Budapest H-1364,
Hungary. Email: keszegh.balazs@renyi.mta.hu.
Research supported by Hungarian National Science Fund (OTKA), under grant PD 108406 and by the J\'anos Bolyai Research Scholarship of the Hungarian Academy of Sciences.}
\and Xuding Zhu\thanks{Department of Mathematics,
Zhejiang Normal University, China. Grant numbers:
NSFC No. 11571319.    Email: xudingzhu@gmail.com}}

\maketitle

\begin{abstract}
This paper studies the choice number and paint number of the lexicographic product of graphs. We prove that if $G$ has maximum degree $\Delta$, then for
any graph $H$ on $n$ vertices $\ch(G[H]) \le (4\Delta+2)(\ch(H) +\log_2 n)$ and $\olch(G[H]) \le (4\Delta+2)
(\olch(H)+ \log_2 n)$.
\end{abstract}

\section{Introduction}

A {\em list assignment } $L$ of a graph $G$ assigns to each vertex
$v$ a set $L(v)$ of permissible colors. An {\em $L$-coloring} of $G$
is a proper vertex coloring of $G$ which colors each vertex with
one of its permissible colors. We say that $G$ is {\em
$L$-colorable} if there exists an $L$-coloring of $G$. Given a
function  $f: V(G) \to \NN$, we say $G$ is   {\em $f$-choosable} if
for every list assignment  $L$ with $|L(v)|=f(v)$ for all $v \in
V(G)$, $G$ is $L$-colorable. We say $G$ is $k$-choosable if for the
constant function $f\equiv k$, $G$ is $f$-choosable.
 The {\em choice number} of $G$ is
$\ch(G)=\min\{k: G$ is $k$-choosable$\}$. List coloring of
graphs has been studied extensively in the literature  (cf.
\cite{Vizing76,ERT79,Tuza97}).

More generally, we say that $G$ is \emph{$(a, b)$-choosable} for some integers $a$ and $b$, $a\ge 2b > 1$, if, for any assignment of lists with $|L(v)| = a$ for all $v \in V$, there are subsets $C(v) \subset L(v)$ with $|C(v)| = b$ such that $C(u)$ and $C(v)$ are disjoint for all pairs of adjacent vertices $u$ and $v$ (the sets $C(v)$ form a \emph{$b$-fold $L$-coloring}). The $b$-choice number of a graph is $\ch_b(G)=\min\{a: G$ is $(a,b)$-choosable$\}$.

Investigation of online variants of combinatorial problems is very common. One reason is that motivations from practical applications and from computer science put emphasis on the algorithmic aspects. Furthermore, besides having direct implications, on-line variants usually also give a deeper insight into the original offline problem. As for many other coloring problems, on-line variant of list coloring came into the attention of researchers.

Assume that $\cup_{v \in V(G)} L(v) = \{1,2,\ldots, q\}$ for some
integer $q$. For $i=1,2,\ldots, q$, let $V_i = \{v: i \in L(v)\}$.
The sequence $(V_1,V_2, \ldots, V_q)$ is   another way of specifying
the list assignment.   An $L$-coloring of $G$ is equivalent to a
sequence $(X_1, X_2, \ldots, X_q)$ of independent sets that form a
partition of $V(G)$ and such that $X_i \subseteq V_i$ for
$i=1,2,\ldots, q$. This point of view of list coloring motivates the
definition of on-line ($b$-fold) list coloring in the form of the following painting game on a graph $G$, which
was introduced in \cite{Schauz2009,Schauz2010b}.

\begin{definition}
	Given a finite graph $G$ and two mappings $f, g: V(G) \to \mathbb{N}$, the
	$g$-fold $f$-painting  game on $G$ is played by two players, Lister and
	Painter. In the $i$-th round, Lister presents a non-empty subset
	$V_i$ of $V(G)$, where each $v \in V_i$ is contained in less than $g(v)$ of the $X_j$'s for $j \le i-1$,
	and  Painter chooses an independent set $X_i$
	contained in $V_i$.  If   at the
	end of some round,  a vertex $v$ is contained in $f(v)$ of the sets $V_i$'s
	but contained in less than $g(v)$ of the sets $X_i$'s, then the
	game ends and Lister wins the game. Otherwise, at some round, each
	vertex $v$ is contained in $g(v)$ of the sets $X_i$, the game ends and Painter  wins the
	game.		
\end{definition}

If $v \in V_i$, then we say color $i$ is a permissible color of vertex $v$. If $v \in X_i$, then we say $v$ is
colored by color $i$. 

 We say $G$ is {\em $g$-fold $f$-paintable} if Painter
has a winning strategy in the $g$-fold $f$-painting game on $G$, and we say
   $G$ is  {\em $b$-fold $k$-paintable} if $G$ is $g$-fold $f$-paintable for the constant functions $f\equiv k$ and $g\equiv b$.
The {\em $b$-paint number} of $G$, denoted by $\chi_{P,b}(G)$, is
the least integer $k$ such that $G$ is $b$-fold $k$-paintable.
The {\em paint number} of $G$, denoted by $\chi_{P}(G)$, is
the least integer $k$ such that $G$ is  $1$-fold $k$-paintable.

Many papers studied the connection of this on-line list coloring with the usual list coloring.
First, it follows from the definition that for any graph $G$, $\olch(G) \ge
\ch(G)$ and $\ch_b(G)\le \cholb(G)$. It was proved in \cite{chip} that the difference $\olch(G) -  \ch(G)$  can be  arbitrarily big
 for complete bipartite graphs. On the other hand, many currently known upper
bounds for the choice numbers of  classes of graphs  remain upper
bounds for their paint number. For example, the paint number of
planar graphs is at most $5$ \cite{Schauz2009}, the paint number of
planar graphs of girth at least $5$ is at most $3$
\cite{Schauz2009,CZ2012}, the paint number of the line graph $L(G)$
of a bipartite graph $G$ is  $\Delta(G)$ \cite{Schauz2009}, the maximum degree of $G$, and if
$G$ has an orientation in which the number of even Eulerian\footnote{A spanning subgraph of a directed graph is 
even (resp. odd) Eulerian if it has an even (resp. odd) number of edges, and for each vertex its  in-degree   equals its  out-degree.}
subgraphs differs from the number of odd Eulerian subgraphs and
$f(x) = d^+(x)+1$, then $G$ is $f$-paintable
\cite{Schauz2010b}.

To avoid confusion we mention that a different list coloring game, 
\emph{game list coloring}, was investigated by Bo\-ro\-wi\-ec\-ki, Sidorowicz  and Tuza \cite{GameCol}. 
In their list coloring game, the graph and the lists are given at the beginning of the game,
and both players colour vertices 
with colours from their lists, however, with opposite goals.  

\bigskip
Various graph products, such as direct product, Cartesian product, strong product and lexicographic product, are important and popular methods of constructing new graphs from old ones. It is interesting to understand how a graph parameter of a product graph
relates to graph parameters of the factor graphs. In this paper, we are interested in the choice number and the paint number of
the lexicographic product of graphs.

\begin{definition}
Let $G=(V_1,E_1)$ and $H=(V_2,E_2)$ be two graphs. The
\textit{lexicographic product} of $G$ and $H$ is the graph $G[H]$
with vertex set $V_1 \times V_2$ and $(v_1,v_2)$ is joined to
$(v_1',v_2')$ if either $(v_1,v_1') \in E_1$ or $v_1=v_1'$ and
$(v_2,v_2') \in E_2$.
\end{definition}


Lexicographic products have a close connection with fractional colorings.
The \emph{fractional chromatic number} of a graph $G$ is defined as $\chi_f(G)=\inf \chi(G[K_n])/n$,
where the infimum is over all $n$.  The \emph{fractional choice number} or
\emph{choice ratio} is defined as $\ch_f(G) = \inf\{\ch_b(G)/b\}=\inf\{a/b : G$
is $(a, b)$-choosable$\}$.
In the same vein, {\em fractional paint number} is defined as $\cholf(G) = \inf\{\cholb(G)/b\}$.
Note that the $b$-choice number can be imagined as a
restricted version of the choice number of the lexicographic product $G[K_b]$,
in which the lists are restricted such that all vertices of a copy of $K_b$ have
the same set of permissible colors. Thus  $\ch_b(G)\le \ch(G[K_b])$. Similarly,
 $\chi_{P,b}(G)\le \chi_P(G[K_b])$.

It follows from the definition that $\chi_f(G) \le \ch_f(G)$ for all graphs $G$.
Although $\ch(G)-\chi(G)$ can be arbitrarily large,
 Alon, Tuza and Voigt
\cite{ATV97} showed the equality  $\ch_f(G) = \chi_f(G)$ holds for all graphs $G$.
This result was further strengthened by Gutowski \cite{G2011}, who showed that for any graph $G$,
its fractional paint number also equals the fractional chromatic number.
About the choice number of lexicographic products our knowledge is much more limited.
In this paper we investigate the choice number of lexicographic products along with its
 on-line variant, the paint number of lexicographic products.

\subsection{Our results and discussion}

First we summarize the trivial relations between the coloring numbers we are interested in:
\begin{obs}
	$\chi(G[H])\le \ch_{\chi(H)}(G)\le \ch(G[H])\le \chol(G[H])$.
\end{obs}

Our aim is to bound these coloring numbers with a function of parameters depending   on $G$ and $H$.
For the chromatic number of lexicographic products, the following is trivially true:

\begin{obs}\label{obs:chrom}
$\chi(G[H])\le \chi(G)\chi(H)$.
\end{obs}

In general this bound cannot be improved as seen by choosing $G=K_r$ and $H=K_n$ (it can be slightly improved in certain cases \cite{GS1975}).

Recall that the $b$-choice number can be regarded as a restricted version of the choice number of the lexicographic product $G[K_b]$ or more generally of $G[H]$ where $\chi(H)=b$.  For the $b$-choice number we do not know a bound as in Observation \ref{obs:chrom}, yet a simple greedy coloring (of copies of $H$ one by one) shows that:
\begin{obs}\label{obs:restr}
	$\ch_b(G)\le b\Delta(G)$.
\end{obs}

Or, equivalently, $\ch_{\chi(H)}(G)\le \Delta(G)\chi(H)$.

For the choice number $\ch(G[H])$ of $G[H]$ no such bound can be expected as shown by $G=K_2$ and $H=E_n$, the empty graph on $n$ vertices. Indeed, for these graphs $\ch(G[H])=\Omega(\log n)$ while $\Delta(G)=\chi(H)=1$.
Still, we are interested in finding a similar upper bound for the choice number of $G[H]$. In order to do that, we study the corresponding problem about the paint number. We present an upper bound on the paint number of the lexicographic product $G[H]$ in terms
of the maximum degree $\Delta(G)$ of $G$, the paint number of $H$ and the
cardinality of $H$. It turns out that the proof can be immediately adjusted for the choice number. This shows that indeed the investigation of the paint number was beneficial in understanding the behaviour of the choice number. Let $\log$ stand for the base $2$ logarithm.

\begin{theorem}\label{thm:main}
For any graph $G$ and any graph  $H$ on $n$ vertices,
\begin{itemize}
\item $\ch(G[H]) \le (4\Delta(G)+2)(\ch(H) +\log n)$,
\item $\olch(G[H]) \le (4\Delta(G)+2)(\olch(H) +\log n)$.
\end{itemize}
\end{theorem}

Note that this theorem generalizes to $b$-fold colorings, for details see Corollary \ref{thm:bfold}.
Applying the theorem with $H=K_b$ we get a similar bound on $\cholb(G)$ as in Observation \ref{obs:restr} we had for $\ch_b(G)$:

\begin{corollary}\label{cor:chb}
	$\cholb(G)\le cb\Delta(G)$ for some constant $c$.
\end{corollary}

Theorem \ref{thm:main} generalizes a result of Alon about complete multipartite graphs (his proof is based on probabilistic arguments). Let $K_{n*r}$ denote the complete $r$-partite graph with $n$ vertices in each vertex class. Applying Theorem \ref{thm:main} with $G=K_r$ and $H=E_n$ (the empty graph on $n$ vertices), we get the upper bound in Alon's following result:

\begin{theorem}\cite{Alon92}\label{thm:alon}
There exist two positive constants $c_1$ and $c_2$ such that for every $n \ge 2$ and for every $r \ge 2$
$$c_1r \log n \le \ch(K_{n*r}) \le c_2 r log\ n.$$
\end{theorem}


Alon's result implies for an arbitrary graph $G$ that $\ch(G)\le c_2 \chi(G) \log m$ where $m$ is the number of vertices of $G$. This was proved also for the paint number \cite{Zhu2009}, showing that $\chol(G)\le \chi(G)\ln m+1$, where $\ln$ stands for the natural logarithm. More generally, about the $b$-paint number it was proved that:

\begin{theorem}\cite{Zhu2009}\label{cholbzhu}
	$\cholb(G)\le c\chi(G)(b+\log m)$ for some constant $c$.
\end{theorem}

Comparing this result with Corollary \ref{cor:chb} we see that while this upper bound depends on $\chi(G)$ and also on $m$, the size of $G$, our upper bound depends only on $\Delta(G)$. Thus as a side effect of our main results we got a bound on $\cholb(G)$ which is best known for many pairs of $G$ and $b$. Note that, e.g., if $b$ is big compared to $\ln m$ then the bound of Theorem \ref{cholbzhu} remains better then the one in Corollary \ref{cor:chb}.

\bigskip
We continue by showing that Theorem \ref{thm:main} is considerably tight. For simplicity, from now on $\Delta$ denotes $\Delta(G)$. First observe that if $G=K_r$ and $H=K_n$ then $G[H]=K_{rn}$ and so $\olch(G[H])=\ch(G[H])=\ch(K_{rn})=rn=(\Delta+1)\ch(H)=(\Delta+1)\olch(H)$.
On the other hand if $G=K_r$ and $H=E_n$, then $G[H]=K_{n*r}$, $\Delta=r-1$ and $\olch(H)=1$ while the lower bound in Theorem \ref{thm:alon} gives $\olch(G[H])\ge \ch(G[H])=\Omega(\Delta \log n)$. Therefore, the $\Delta(G) \ch(H)$ (resp. $\Delta(G) \olch(H)$) term and the $\Delta(G)\log n$ term in Theorem \ref{thm:main} are both needed, the theorem is tight up to a constant multiplicative factor.

Nevertheless, inspired by the bound of Observation \ref{obs:chrom}, we ask if bounds similar to the ones in Observation \ref{obs:restr} and Theorem \ref{thm:main} hold in which the dependence on $\Delta$ is replaced by a dependence on $\ch(G)$ or $\olch(G)$:

\begin{problem}\label{prob:ch}
	Does there exist a constant $c$ such that	for any graph $G$ and any graph $H$ on $n$ vertices,
	\begin{enumerate}
		\item[(i)] $\ch_b(G) \le c b\ch(G)$,
		\item [(ii)]$\ch(G[H]) \le c\ch(G)(\ch(H) +\log n)$,
		\item [(iii)]$\olch(G[H]) \le c\olch(G)(\olch(H) +\log n)$?
	\end{enumerate}
\end{problem}

 We note that a positive answer to  Problem \ref{prob:ch}(i) is implied by the following well known conjecture of Erd\H os, Rubin and Taylor:

\begin{conjecture}\cite{ERT79}\label{prob:ert}
	If $G$ is $(a,b)$-choosable, then $G$ is $(an ,bn)$-choosable for every $n\in \NN$.
\end{conjecture}

Indeed, Conjecture \ref{prob:ert} for $b=1$ is equivalent to Problem \ref{prob:ch}(i) with $c=1$ (note that $n$ in Conjecture \ref{prob:ert} plays the role of $b$ in Problem \ref{prob:ch}(i)).
However,   to   our knowledge,  it is already unknown whether or not  there exists a polynomial $p$ such that $\ch_b(G) \le p(b)\ch(G)$.

\section{An upper bound for $\ch(G[H])$ and $\olch(G[H])$}

We denote by $N_G(x)$ the open neighborhood of $x$ in $G$,
i.e., the set containing all the neighbors of $x$ in $G$, and by $N_G[x]$ the closed neighborhood of $x$ in $G$,
i.e., $N_G[x]=N_G(x)\cup \{x\}$.
In $G[H]$,  for  $x\in V(G)$, let
 $$V_x = \{(x,y):y\in H\},$$
 which is called a {\em layer} of $G[H]$. Observe that each $V_x$ induces a copy of $H$.
For $C \subset G[H]$, let  $$C(x) = C \cap V_x,$$
and call $C(x)$  a {\em layer of $C$}.
Let $E_n$ be the empty graph (i.e., edgeless) on $n$ vertices.

In this section, we first give an upper bound for $\chi_{P,b}(G[E_n])$.
Then we use this result to prove Theorem \ref{thm:main}. Note that Corollary \ref{cor:chb} already follows from the following lemma by setting $n=1$.

\begin{lemma}\label{lemma:main}
	For any graph  $G$ and positive integer $n$,    $\chi_{P,b}(G[E_n]) \le (4\Delta(G)+2)(b +\log n)$.
\end{lemma}

\begin{proof}
The idea of the proof is similar to the idea used in \cite{Zhu2009}.
Let
$$k= \lceil (4\Delta+2)(b +\log n) \rceil.$$ We shall give a winning strategy for Painter
in the $b$-fold $k$-painting game on $G[E_n]$.

During the game, Painter will keep track of a weighting function, which is dynamic as the weight of
vertices will change during the game. We shall denote by  $h_i(v)$  the   weight of $v$ at the end of the $i$th round.

For a set $X$ of vertices of $G[E_n]$, let  $$h_i(X)  = \sum_{v \in X}h_i(v).$$

At the beginning no vertex is colored and each vertex has weight $h_0(v)=1$.

Assume that in round $i$,  Lister presents a set of vertices, $C$, all colored less than $b$ times.
We construct an independent set $I$ of $G$ recursively as follows:

Initially $I = \emptyset$ and $U=V(G)$.

Repeat the following step
until $U$ becomes empty: choose a vertex $x \in U$ with $$h_{i-1}(C(x)) =
\max \{h_{i-1}(C(y)): y \in U\},$$ let $$I:=I \cup \{x\}, \ U := U-
N_G[x].$$
When this procedure of choosing greedily maximal weight layers of $C$ ends, $I$ is a maximal independent set of $G$.
It follows from the construction that for any $x'\notin I$, there is a
vertex $x \in I\cap N_G(x')$ such that $h_{i-1}(C(x)) \ge h_{i-1}(C(x'))$.

Painter colors the vertices in $\cup_{x \in I}C(x)$. The weight function is
updated as follows:
\[
h_i(v)=
\begin{cases}(1+ \frac {1}{4\Delta+1}) h_{i-1}(v),\quad \ \ &
  v \in C(x), x \notin I,\\
  \frac {1}{2} h_{i-1}(v),\quad \ \ &
  v \in C(x), x \in I,\\
  h_{i-1}(v), \quad \ \ & v \notin C,
\end{cases}
\]

We shall prove that this is a winning strategy for Painter.

We first show that for any vertex $x$ of $G$, $h_i(V_x) \le 2n$.

Assume $x$ is a fixed vertex of $G$. We   construct recursively an infinite rooted tree $T$ and a homomorphism $\phi$ from $T$ to $G$. Initially, let $T_0$ be the tree with a single root vertex $r$. Let $\phi_0(r)=x$.
Assume $T_i$ and a homomorphism $\phi_i: T_i \to G$ have been constructed. We construct $T_{i+1}$ and $\phi_{i+1}$ as follows:
For each leaf vertex $a$ of $T_i$, add $d_G(\phi_i(a))-1$ leaves ($d_G(\phi_i(a))$ leaves in the case $i=0$) adjacent to $a$, and extend $\phi_i$ to these leaves
in such a way that $\phi_{i+1}(N_{T_{i+1}}(a)) = N_G(\phi_i(a))$. In other words, the restriction of $\phi_{i+1}$ to $N_{T_{i+1}}(a)$ is
a bijection from $N_{T_{i+1}}(a)$ to $N_G(\phi_i(a))$.

Observe that $T_i$ is a subtree of $T_{i+1}$ and the homomorphism $\phi_i$ of $T_i$ to $G$ is extended to a homomorphism $\phi_{i+1}$ of $T_{i+1}$ to $G$.
Let $T = \lim_{i \to \infty}T_i$, and let $\phi = \lim_{i \to \infty}\phi_i$. It follows from the definition that for any vertex $a$ of $T$,
$d_T(a) = d_G(\phi(a))$ and $\phi$ is a bijection between $N_T(a)$ and $N_G(\phi(a))$. For each vertex $a$ of $T$, the {\em level} $l(a)$ of $a$ is the
distance from $a$ to the root $r$. In particular, $l(r)=0$. For $a \in T$, let
\[
\eta(a)=
 \frac{1}{(2\Delta)^{l(a)}}.
 \]

Let $$w_i(T) = \sum_{a \in T} \eta(a)h_i(V_{\phi(a)}).$$

Notice that by definition $h_i(V_x)\le w_i(T)$: when $a$ is the root of the tree, then   the corresponding summand is $h_i(V_x)$ .
Since the maximum degree of $T$ is equal to the maximum degree of $G$, we know that the number of vertices of level $l$ is at most $\Delta (\Delta-1)^{l-1} < \Delta^l$.
As $h_0(V_y)=n$ for any vertex $y$ of $G$, we have
$$w_0(T) \le n  \sum_{l=0}^{\infty}  \Delta^{l}  \frac {1}{(2\Delta)^l}  =2n.$$

Next we show that for $i \ge 0$, $w_{i+1}(T) \le w_i(T)$.

Let
\begin{eqnarray*}
A &=& \{a \in V(T): h_{i+1}(V_{\phi(a)}) > h_i(V_{\phi(a)})\}\\
B &=& \{a \in V(T): h_{i+1}(V_{\phi(a)}) < h_i(V_{\phi(a)})\}.
\end{eqnarray*}

For simplicity, for $a \in T$, let $h_i(a) = h_i(C(\phi(a)))$.
It follows from the definition that \begin{eqnarray*}
&& \forall a \in A, h_{i+1}(V_{\phi(a)}) - h_i(V_{\phi(a)}) =    \frac {1}{4\Delta+1}h_i(a),\\
&& \forall b \in B, h_{i+1}(V_{\phi(b)}) - h_i(V_{\phi(b)}) =  - \frac 12 h_i(b).
\end{eqnarray*}

\begin{claim} For each $a \in A$, there exists $b \in B$ such that $a,b$ are adjacent in $T$ and $h_i(a) \le h_i(b)$.
\end{claim}
\begin{proof}
Let $x =\phi(a)$. By our choice of $I$,   $x$ has a neighbor $x' \in I$ for which $h_i(C(x')) \ge h_i(C(x))$.
By our construction of $T$, $a$ has a unique neighbor $b$ with $\phi(b)=x'$.
\end{proof}

Continuing the proof of Lemma \ref{lemma:main}, for $b \in B$, let $$A_b=\{a \in A: ab \in E(T), h_i(a) \le h_i(b)\}.$$

\begin{eqnarray*}
 w_{i+1}(T) - w_i(T)
& =& \sum_{a \in A \cup B}\eta(a)\left(h_{i+1}(V(\phi(a)))-h_i(V(\phi(a)))\right)\\
&=&
\frac {1}{4\Delta+1}\sum_{a \in A}h_i(a)\eta(a)  -\frac 12 \sum_{b\in B}h_i(b) \eta(b)\\
&\le &\sum_{b \in B} \left(\frac {1}{4\Delta+1}\sum_{a \in A_b}h_i(a)\eta(a) - \frac12 h_i(b) \eta(b)\right) \\
&\le & \sum_{b \in B}h_i(b) \left(\frac {1}{4\Delta+1}\sum_{a \in A_b}\eta(a)  -  \frac12 \eta(b)\right)
\end{eqnarray*}

We shall show that $\frac {1}{4\Delta+1}\sum_{a \in A_b}\eta(a)  - \frac12 \eta(b) \le 0$ for all $b \in B$, which will imply that $w_{i+1}(T) - w_i(T) \le 0$.
There is at most one $a' \in A_b$ for which $l(a')=l(b)-1$.
For all other vertices $a \in A_b$, we have $l(a) = l(b)+1$. Thus
$\eta(a') = 2\Delta\eta(b)$, and
$\eta(a) = \frac{1}{2\Delta}\eta(b)$ for $a \in A_b - \{a'\}$ (the vertex $a'$ may not exist).
Therefore,
\begin{eqnarray*}
\frac {1}{4\Delta+1}\sum_{a \in A_b}\eta(a)  - \frac12 \eta(b) &\le& \eta(b)\left( \frac {\Delta-1}{4\Delta+1}   \frac{1}{2\Delta}+ \frac{2\Delta}{4\Delta+1} -      \frac12\right) \\
&\le&0
\end{eqnarray*}
Therefore $w_{i+1}(T) \le w_i(T)$. Hence
$$h_i(V_x) \le w_i(T) \le w_0(T) =2n.$$

Assume $v \in V(G[E_n])$ is contained in layer $V_x$. If at the end of the $i$th round, $v$ has been given $s$ permissible colors and is colored
with $t \le b-1$ colors, then
$$h_i(v) =\left(1+ \frac{1}{4\Delta+1}\right)^{s-t}\left(\frac{1}{2}\right)^t \le h_i(V_x) \le 2n.$$
Hence
$$s \le t+\log_{1+\frac{1}{4\Delta+1}}(2^t2n)=t+\frac{\log(2^t2n)}{\log(1+\frac{1}{4\Delta+1})}
\le b-1+(4\Delta+1)
( b-1+\log 2n)= $$ $$=(4\Delta+2)(b +\log n)- \log n -1\le k-1.$$
(The second inequality above follows from the fact that $1/\log(1+1/x)\le x$ as $2\le(1+1/x)^x$ for $x=4\Delta+1$.)
In other words, if    $k$ permissible  colors have been assigned to a vertex $v$, then $v$ would have been colored $b$ times.
Hence the strategy is   a winning strategy for Painter.
This completes the proof of Lemma  \ref{lemma:main}.
\end{proof}


Now we are ready to prove Theorem \ref{thm:main}.

\begin{proof}[Proof of Theorem \ref{thm:main}.]
Recall that we want to prove that assuming $G$ is a graph of maximum degree $\Delta$, then for any graph $H$ on $n$ vertices,
 $\ch(G[H]) \le (4\Delta(G)+2)(\ch(H) +\log n)$ and
$\olch(G[H]) \le (4\Delta+2)(\olch(H) +\log n)$.

Let   $b=\ch(H)$ and  $k = (4\Delta+2) (b +\log n) $. For any list assignment $L$ of $G[H]$ with lists of size $k$ we shall present an $L$-coloring of $G[H]$.

By Lemma \ref{lemma:main} $\ch_b(G[E_n])\le\cholb(G[E_n])\le k$ and so there exists a $b$-fold $L$-coloring of $G[E_n]$. Now this coloring assigns $b$ colors to each vertex of $G(H)$. Define on the vertices of each layer $V_x$ ($x\in G$) a list $L_x$ where the list of $v\in V_x$ consists of the $b$ colors assigned to $v$ in this $L$-coloring. As $b=\ch(H)$, each copy $G[V_x]$ of $H$ can be $L_x$-colored. The list colorings $L_x$ ($x\in G$) assign a color to each vertex, let this be our final coloring. This coloring by definition assigns a permissible color (by $L$) to each vertex, furthermore by the first round we made sure that if two vertices in different layers are connected then they are colored differently. Finally, in the second round we made sure that if two vertices in the same layer are connected then they are colored differently. Thus this is a proper coloring, and so it is an $L$-coloring, as required.

The proof of the second part of the theorem is similar. Let $b=\olch(H)$ and $k = (4\Delta+2) (b +\log n) $. We shall present a winning
strategy for Painter in the $k$-painting game on $G[H]$.

Assume that in a certain round Lister presents a set $C$ of uncolored vertices
of $G[H]$. Painter's response is determined in  two steps.
  In the first step,   Painter plays the $b$-fold $k$-painting game on $G[E_n]$, using the strategy in the proof of Lemma
\ref{lemma:main}. Assume $I$ is the independent set of $G[E_n]$ chosen by Painter in the first step.
In the second step, for each vertex $x$ of $G$, Painter plays the $b$-painting game on $G[V_x]$, by assuming that
Lister has chosen the set $I \cap V_x$ in this round, using his winning strategy for this game (which exists as $b=\olch(H)$).
Assume for each vertex $x$ of $V(G)$, Painter's choice (in the second step) is $A_x$.
For the $k$-painting game on $G[H]$, Painter's response in this round is $\cup_{x \in V(G)}A_x$.

First we prove that $\cup_{x \in V(G)}A_x$ is an independent set of $G[H]$ contained in $C$. Indeed, as $I$ is an independent set of $G[E_n]$ contained in $C$, we conclude that
the set $\{x: I \cap V_x \ne \emptyset\}$ is an independent set of $G$.  As each $A_x$ is an independent set of $G[H]$ contained in
$V_x \cap I$, we conclude that $\cup_{x \in V(G)}A_x$ is an independent set of $G[H]$ contained in $C$.

It remains to show that every vertex of $G[H]$ will eventually be colored. Assume $v \in V_x$. As $v$ has $k$ permissible colors,
during the first steps, $v$ will be given $b$  colors, i.e.,
$v$ is contained in $b$ of $I$'s that are chosen by Painter in the
first steps. That means for the $b$-painting game on $G[V_x]$, $v$ will be given $b$ permissible colors. As Painter is using a winning strategy for the $b$-painting
game on $G[V_x]$, $v$ will eventually be colored.  Therefore the above strategy is a winning strategy for Painter on the $k$-painting game on $G[H]$ and hence
$\olch(G[H]) \le (4\Delta+2)
(\olch(H)+ \log n)$.
\end{proof}

This proof easily generalizes also for the case of $b$-paint number:

\begin{corollary}\label{thm:bfold}
\label{cor:main} Assume $G$ is a graph of maximum degree $\Delta$. Then
for any graph $H$ on $n$ vertices,
\begin{itemize}
\item $\ch_{b}(G[H]) \le (4\Delta(G)+2) (\ch_{b}(H)+ \log n)$,
\item $\chi_{P,b}(G[H]) \le (4\Delta(G)+2) (\chi_{P,b}(H)+ \log n)$.
\end{itemize}

\end{corollary}
\subsubsection*{Acknowledgment} This research was initiated while the first author enjoyed the hospitality of Zhejiang Normal University.

\bibliography{choosability}

\bibliographystyle{plain}

\end{document}